\documentclass[a4paper,10pt]{amsart}
\usepackage[utf8]{inputenc}
\usepackage[T1]{fontenc}

\usepackage{amsmath}
\usepackage{amssymb}
\usepackage{amsthm}
\usepackage{mathtools}

\usepackage{eucal} 
\usepackage{mathdots}
\usepackage[colorlinks=true,linkcolor=red,citecolor=blue,urlcolor=green]{hyperref}
\usepackage[capitalise]{cleveref}
\usepackage{enumitem}
\usepackage{mathrsfs}  

\usepackage{tikz-cd}
\usepackage{adjustbox}
\usepackage{contour}
\usepackage{ulem}

\numberwithin{equation}{section}
\crefname{equation}{}{}

\crefformat{section}{\S#2#1#3}
\crefmultiformat{section}{\S\S#2#1#3}{ and~#2#1#3}{, #2#1#3}{, and~#2#1#3}

\contourlength{0.8pt}

\newcommand{\ULL}[1]{%
  \uline{\phantom{#1}}%
  \llap{\contour{white}{#1}}%
}

\newtheorem{theorem}{Theorem}[section]
\newtheorem{prop}[theorem]{Proposition}
\newtheorem{cor}[theorem]{Corollary}
\newtheorem{lemma}[theorem]{Lemma}

\theoremstyle{definition}
\newtheorem{dfn}[theorem]{Definition}
\newtheorem{rmk}[theorem]{Remark}
\newtheorem{quest}[theorem]{Question}

\DeclareMathOperator{\Hom}{\mathsf{Hom}}

\DeclareMathOperator{\Ext}{\mathsf{Ext}}
\DeclareMathOperator{\Tor}{\mathsf{Tor}}

\DeclareMathOperator{\RHom}{\mathbf{R}\mathsf{Hom}}

\newcommand{\Acal}{\mathcal{A}}

\newcommand{\Ccal}{\mathcal{C}}

\newcommand{\Fcal}{\mathcal{F}}
\newcommand{\Gcal}{\mathcal{G}}

\newcommand{\Ical}{\mathcal{I}}

\newcommand{\Mcal}{\mathcal{M}}

\newcommand{\Pcal}{\mathcal{P}}

\newcommand{\Scal}{\mathcal{S}}
\newcommand{\Tcal}{\mathcal{T}}

\newcommand{\Xcal}{\mathcal{X}}
\newcommand{\Ycal}{\mathcal{Y}}
\newcommand{\Zcal}{\mathcal{Z}}

\newcommand{\Qbb}{\mathbb{Q}}

\newcommand{\Zbb}{\mathbb{Z}}

\newcommand{\D}{\mathsf{D}}

\newcommand{\Mod}[1]{\mathsf{Mod}\mbox{-}#1}
\newcommand{\lMod}[1]{#1\mbox{-}\mathsf{Mod}}

\renewcommand{\mod}[1]{\mathsf{mod}\mbox{-}#1}

\newcommand{\depth}{\mathsf{depth}}
\newcommand{\width}{\mathsf{width}}
\newcommand{\CH}{\mathsf{Ch}}

\newcommand{\Spec}[1]{\mathsf{Spec}(#1)}

\newcommand{\Supp}{\mathsf{Supp}}
\newcommand{\supp}{\mathsf{supp}}

\newcommand*{\Perp}[1]{{}^{\perp_{#1}}}

\newcommand{\Prod}{\mathsf{Prod}}

\newcommand{\Add}{\mathsf{Add}}

\newcommand{\pp}{\mathfrak{p}}
\newcommand{\PP}{\mathfrak{P}}
\newcommand{\qq}{\mathfrak{q}}
\newcommand{\mm}{\mathfrak{m}}

\newcommand{\ef}{\mathsf{f}}

\newcommand{\height}{\mathsf{height}}

\newcommand{\RGamma}{\mathbf{R}\Gamma}
\newcommand{{\tst}}{\textit{t}-}
\newcommand{\pd}{\mathsf{pd}}
\newcommand{\fd}{\mathsf{fd}}
\newcommand{\id}{\mathsf{id}}

\newcommand{\Gfd}{\mathsf{Gfd}}
\newcommand{\Gid}{\mathsf{Gid}}

\newcommand{\CMHomid}{\mathsf{CM_{\Hom}id}}
\newcommand{\CMDid}{\mathsf{CMid}}
\newcommand{\CMSid}{\mathsf{CM_*id}}
\newcommand{\CMSfd}{\mathsf{CM_*fd}}
\newcommand{\Rid}{\mathsf{Rid}}
\newcommand{\Rfd}{\mathsf{Rfd}}
\newcommand{\rid}{\mathsf{rid}}
\newcommand{\Findim}{\mathsf{Findim}}
\newcommand{\findim}{\mathsf{findim}}

\newcommand{\grade}{\mathsf{grade}}
\newcommand{\newterm}[1]{\ULL{\text{#1}}}

\title{Restricted injective dimensions over Cohen-Macaulay rings}
\author{Michal Hrbek}
\address[M. Hrbek]{Institute of Mathematics of the Czech Academy of Sciences, \v{Z}itn\'{a} 25, 115 67 Prague, Czech Republic}
\email{hrbek@math.cas.cz}

\author{Giovanna Le Gros}
\address[G. Le Gros]{Dipartimento di Matematica ``Tullio Levi-Civita'', Universit\`{a} di Padova, Via Trieste 63, 35121 Padova (Italy)}
\email{legros@math.unipd.it}

\subjclass[2020]{}

\thanks{The first author was supported by the GAČR project 23-05148S and the Academy of Sciences of the Czech Republic (RVO 67985840). This project was partially developed during the visit of the first author to Università degli Studi di Padova, he would like to thank the Dipartimento di Matematica for their hospitality. The visit was partially funded by DFG (Deutsche Forschungsgemeinschaft) through a scientific network on silting theory.}
\begin{document}
\begin{abstract}
  We show that the small and large restricted injective dimensions coincide for Cohen-Macaulay rings of finite Krull dimension. Based on this, and inspired by the recent work of Sather-Wagstaff and Totushek, we suggest a new definition of Cohen-Macaulay Hom injective dimension. We show that the class of Cohen-Macaulay Hom injective modules is the right constituent of a perfect cotorsion pair. Our approach relies on tilting theory, and in particular, on the explicit construction of the tilting module inducing the minimal tilting class recently obtained in \cite{HNS}.
\end{abstract}
\maketitle
\tableofcontents
\section{Introduction}
While the global dimension of a commutative noetherian ring $R$ is infinite unless $R$ is regular, Raynaud and Gruson \cite{RG71} proved that the finitistic global dimension of $R$ is equal to its Krull dimension, which we shall assume to be finite for the rest of this introduction. Christensen, Foxby, and Frankild \cite{CFF02} defined the (large) restricted injective dimension in terms of Ext-orthogonality to the class $\Pcal$ of modules of finite projective dimension, an invariant always bounded by the finitistic global dimension of $R$. One cannot expect Baer's criterion to hold for the restricted injective dimension in general, and thus \cite{CFF02} also considers the small restricted injective dimension defined using only those modules of finite projective dimension which are finitely generated. The situation in which the small and the large restricted injective dimensions coincide occurs precisely when the cotorsion pair $(\Pcal,\Pcal\Perp{1})$ is of finite type (see \cref{ss:cotpairs} and \cref{ss:findim}). The question of when the modules of bounded projective dimension are of finite type was studied in a more general setting by Bazzoni and Herbera \cite{BH09}. In particular, they provide a criterion for when the modules of projective dimension at most one are of finite type.

Our first main result \cref{T:finite-type} shows that $(\Pcal,\Pcal\Perp{1})$ is of finite type precisely when the ring $R$ is Cohen-Macaulay. In addition, in this case the two restricted injective dimensions coincide in this case with the Chouinard invariant, a refinement of injective dimension introduced in \cite{Cho76}. Our approach relies on tilting theory: The class of modules of small restricted injective dimension at most zero is precisely the minimal tilting class in $\Mod R$. In the recent work \cite{HNS}, the tilting module for this tilting class has been constructed explicitly as the coproduct of local cohomology modules. Using this construction, we show in \cref{filtration} that modules from $\Tcal$ admit ``canonical filtrations'', a deconstruction result established first for Gorenstein injectives over Gorenstein rings by Enochs and Huang \cite{E11}. These filtrations then allow us to prove that this tilting module is the unique product-complete tilting $R$-module up to equivalence (\cref{T-prod-compl,T:unique-prodcomp}), and in turn that the left constituent $\Acal$ of the minimal tilting cotorsion pair $(\Acal,\Tcal)$ coincides with $\Pcal$. 

The minimal tilting class $\Tcal$ coincides with the class $\Ical_0$ of all injective $R$-modules if and only if $R$ is regular, and it coincides with the class $\Gcal\Ical_0$ of all Gorenstein injective $R$-modules if and only if $R$ is Gorenstein. The notion of CM-dimension for finitely generated modules, extending the notion of G-dimension of Auslander and Bridger, was introduced by Gerko \cite{G01}. Later, Holm and J{\o}rgensen \cite{HJ07} developed notions of Cohen-Macaulay projective, injective, and flat dimensions in terms of trivial extensions over semidualizing modules. Finiteness of any of these dimensions characterizes Cohen-Macaulay rings admitting a dualizing module. Recently, Sahandi, Sharif, and Yassemi \cite{SSY20} defined other notions of Cohen-Macaulay injective and flat dimensions, whose finiteness characterizes general Cohen-Macaulay rings. Inspired by the complete intersection Hom injective dimension recently introduced by Sather-Wagstaff and Totushek \cite{SWT21}, we define the Cohen-Macaulay Hom injective dimension and prove that it yields a refinement of the notion of Cohen-Macaulay injective dimension of Holm and J{\o}rgensen. Applying our main result on restricted injective dimensions over a Cohen-Macaulay ring, we show that the class $\Ccal\Mcal\Ical_0$ of Cohen-Macaulay Hom injective modules enjoys similar properties as the class $\Gcal\Ical_0$ of Gorenstein injectives over a Gorenstein ring:
\begin{theorem}
  Let $R$ be a Cohen-Macaulay ring of finite Krull dimension. Then the following hold:
  \begin{itemize}
    \item[(i)] [\cref{TCMid}, \cref{T:finite-type}] The minimal tilting cotorsion pair in $\Mod R$ is of the form $(\Pcal,\Ccal\Mcal\Ical_0)$.
    \item[(ii)] [\cref{T-prod-compl}, \cref{cm-flat}] The class $\Ccal\Mcal\Ical_0$ is definable and enveloping. The dual definable class is $\Ccal\Mcal\Fcal_0$ of Cohen-Macaulay flat modules. 
    \item[(iii)][\cref{filtration}] Modules from $\Ccal\Mcal\Ical_0$ admit canonical filtrations.
  \end{itemize}
\end{theorem}
The structure of the paper is as follows. In \cref{s:preliminaries} we gather preliminary facts about restricted injective dimensions and their relation to tilting theory, first over general rings and then specialize to the commutative noetherian situation. The main result establishing the finite type of $\Pcal$ over a Cohen-Macaulay ring is proved in \cref{s:tiltingcm}. In \cref{s:cmhomdim} we introduce our definition of Cohen-Macaulay Hom injective dimension and show that the minimal tilting class over a Cohen-Macaulay ring consists precisely of Cohen-Macaulay Hom injectives. In the final \cref{s:cotilting}, we explain that analogous, and in fact easier results hold on the dual cotilting side and show that the Cohen-Macaulay flat modules form a dual definable class to Cohen-Macaulay Hom injectives.
\section{Preliminaries}\label{s:preliminaries}
Let $R$ be an associative unital ring. We denote by $\Mod R$ the category of all right $R$-modules and by $\mod R$ be the (full, isomorphism-closed) subcategory consisting of those modules which admit a resolution by finitely generated projective $R$-modules. For any module $M$, let $\Add(M)$ denote the subcategory consisting of all modules which are isomorphic to a direct summand of the coproduct $M^{(X)}$ for some set $X$. Similarly, $\Prod(M)$ is the subcategory consisting of all modules which are isomorphic to a direct summand of the product $M^{X}$ for some set $X$. We let $\D(R)$ denote the unbounded derived category of cochain complexes of right $R$-modules. Given $M \in \D(R)$, we let $\inf M = \inf \{n \in \Zbb \mid H^n(M) \neq 0\}$ and $\sup M = \sup \{n \in \Zbb \mid H^n(M) \neq 0\}$ denote the cohomological infimum and supremum of $M$. We call $M$ \newterm{cohomologically bounded} if either $M = 0$ in $\D(R)$ or if both $\inf M$ and $\sup M$ are integers. We use the usual notation $\Ext_R^i(M,N) = H^i \RHom_R(M,N)$ and $\Tor_i^R(M,N) = H^{-i}(M \otimes_R^\mathbf{L} N)$ for any $M$ and $N$ cochain complexes of right (resp., left) $R$-modules.

\subsection{}\label{ss:dimensions} For $n \geq 0$, let $\Pcal_n = \{M \in \Mod R \mid \pd_R M \leq n\}$, $\Ical_n = \{M \in \Mod R \mid \id_R M \leq n\}$, and $\Fcal_n = \{M \in \Mod R \mid \fd_R M \leq n\}$ denote the subcategories of $\Mod R$ consisting of all modules of projective, injective, or flat dimension bounded above by $n$. We use the notation $\Pcal = \bigcup_{n \geq 0}\Pcal_n$ for modules of finite projective dimension, similarly we put $\Fcal = \bigcup_{n \geq 0}\Fcal_n$ and $\Ical = \bigcup_{n \geq 0}\Ical_n$. Furthermore, we let $\Pcal_n^f = \Pcal_n \cap \mod R$ and $\Pcal^f = \Pcal \cap \mod R$. 

Christensen, Foxby, and Frankild introduced in \cite{CFF02} the notions of restricted homological dimensions for a commutative noetherian ring. The following is one of the ways of extending their definition to an arbitrary ring. Given an $R$-module or an $R$-complex $M$, we define the \newterm{small restricted injective dimension} as
$$\rid_R(M) = \sup\{i \mid \Ext_R^i(\Pcal^f, M)\neq 0\}$$
and the \newterm{(large) restricted injective dimension} is defined as
$$\Rid_R(M) = \sup\{i \mid \Ext_R^i(\Pcal, M)\neq 0\}.$$
As with other notions of homological dimensions, the definition entails the convention $\Rid(0) = -\infty = \rid(0)$. 
\subsection{}\label{ss:cotpairs} Given a subcategory $\Ccal$ of $\Mod R$ we use the notation $\Ccal\Perp{1}= \{M \in \Mod R \mid \Ext_R^1(C,M) = 0 ~\forall C \in \Ccal\}$ and $\Ccal\Perp{}= \{M \in \Mod R \mid \Ext_R^i(C,M) = 0 ~\forall C \in \Ccal, ~\forall i>0\}$; we also define $\Perp{1}\Ccal$ and $\Perp{}\Ccal$ analogously and we drop the curly brackets whenever $\Ccal = \{C\}$ for some module $C$. A \newterm{cotorsion pair} in $\Mod R$ is a pair $(\Xcal,\Ycal)$ of subcategories of $\Mod R$ such that $\Ycal = \Xcal\Perp{1}$ and $\Xcal = \Perp{1}\Ycal$. A cotorsion pair is \newterm{complete} if for any $M \in \Mod R$ there are exact sequences $0 \to M \to Y^M \to X^M \to 0$ and $0 \to Y_M \to X_M \to M \to 0$ with $X^M,X_M \in \Xcal$ and $Y^M,Y_M \in \Ycal$. It is \newterm{hereditary} if $\Ext_R^i(X,Y) = 0$ for all $X \in \Xcal$, $Y \in \Ycal$, and $i>0$. Finally, it is \newterm{of finite type} if there is $n \geq 0$ and a subset $\Scal$ of $\Pcal_n^f$ such that $\Ycal = \Scal\Perp{}$. Any cotorsion pair of finite type is automatically complete and hereditary, see \cite[\S 6]{GT12}.
\subsection{}\label{ss:findim} Recall that the \newterm{finitistic dimension} of $R$ is defined as 
$$\Findim(R) = \sup \{\pd_R M \mid M \in \Pcal\},$$
while the \newterm{small finitistic dimension} of $R$ is 
$$\findim(R) = \sup \{\pd_R M \mid M \in \Pcal^f\}.$$
We have $\Findim(R) \leq n < \infty$ if and only if $\Pcal = \Pcal_n$ and $\findim(R) \leq n < \infty$ if and only if $\Pcal^f = \Pcal^f_n$. Clearly, we have $\Rid_R(M) \leq \Findim(R)$ and $\rid_R(M) \leq \findim(R)$ for all $M \in \Mod R$. If $\Findim(R)<\infty$, then \cite{AEJO} implies that $(\Pcal,\Pcal\Perp{1})$ is a complete hereditary cotorsion pair. If $\findim(R)<\infty$ then $(\Perp{1}(\Pcal^f\Perp{1}),\Pcal^f\Perp{1})$ is a cotorsion pair of finite type. Assuming $\Findim(R)<\infty$ ($\implies \findim(R)<\infty$), it follows that both $\Rid_R$ and $\rid_R$ are relative cohomological dimensions induced by complete hereditary cotorsion pairs. In particular, an $R$-module or a cohomologically bounded $R$-complex $M$ satisfies $\Rid_R(M) \leq k$ (resp. $\rid_R(M) \leq k$) if and only if it admits a $\Pcal\Perp{1}$-coresolution (resp., $\Pcal^f\Perp{1}$-coresolution) of length $k$, that is, there is an exact sequence
$$0 \to M \to P^0 \to P^1 \to \cdots \to P^k \to 0$$
with $P^i \in \Pcal\Perp{1}$ (resp., $P^i \in \Pcal^f\Perp{1}$) for all $i=0,1,\ldots,k$.
We have the following observation:
\begin{lemma}\label{l:fintype}
  Assume $\Findim(R) < \infty$, then the following are equivalent:
  \begin{enumerate}
    \item[(i)] $\Rid_R(M) = \rid_R(M)$ for each cohomologically bounded complex $M$, 
    \item[(ii)] $\Rid_R(M) = \rid_R(M)$ for each $R$-module $M$,
    \item[(iii)] the cotorsion pairs $(\Pcal,\Pcal\Perp{1})$ and $(\Perp{1}(\Pcal^f\Perp{1}),\Pcal^f\Perp{1})$ coincide,
    \item[(iv)] the cotorsion pair $(\Pcal,\Pcal\Perp{1})$ is of finite type.
  \end{enumerate}
\end{lemma}
\begin{proof}
  Since $\findim(R) \leq \Findim(R) < \infty$, both $(\Perp{1}(\Pcal^f\Perp{1}),\Pcal^f\Perp{1})$ and $(\Pcal,\Pcal\Perp{1})$ are complete hereditary cotorsion pairs.

  $(i) \implies (ii)$: Trivial.

  $(ii) \implies (iii)$: This follows directly from $\Pcal\Perp{1} = \{M \in \Mod R \mid \Rid_R(M) \leq 0\}$ and $\Pcal^f\Perp{1} = \{M \in \Mod R \mid \rid_R(M) \leq 0\}$.

  $(iii) \implies (iv):$ Trivial, as $(\Perp{1}(\Pcal^f\Perp{1}),\Pcal^f\Perp{1})$ is of finite type.

  $(iv) \implies (i)$: The assumption $(iv)$ implies that $\Rid_R(M) \leq 0$ if and only if $\rid_R(M) \leq 0$. As noted above, for a cohomologically bounded complex $M$, the dimensions $\Rid_R(M)$ and $\rid_R(M)$ can be computed by taking coresolutions by modules with the respective dimension being equal to zero.
\end{proof}
The question of when the cotorsion pair $(\Pcal_n,\Pcal_n\Perp{1})$ is of finite type was studied by Bazzoni and Herbera \cite{BH09}. In particular, this can fail to be true even for $n=1$ and $R$ commutative noetherian \cite[Theorem 8.6]{BH09}. In addition, by the Eklof-Trlifaj theorem \cite[6.2, 6.14]{GT12}, the cotorsion pair $(\Pcal,\Pcal\Perp{1})$ is of finite type if and only if any module $M \in \Pcal$ is a direct summand of a module filtered (=obtained as a transfinite extension) by modules from $\Pcal^f$.
\subsection{}\label{ss:tilting} A module $T \in \Mod R$ is called a \newterm{tilting module} if the following three conditions are satisfied:
\begin{enumerate}
  \item[(T1)] $T \in \Pcal$,
  \item[(T2)] $\Add(T) \subseteq T\Perp{}$,
  \item[(T3)] there is $n \geq 0$ and a short exact sequence $0 \to R \to T_0 \to T_1 \ldots \to T_n \to 0$ with $T_i \in \Add(T)$ for each $i=0,1,\ldots,n$.
\end{enumerate}
Any tilting module $T$ gives rise to a complete hereditary cotorsion pair $(\Acal,\Tcal)= (\Perp{1}(T\Perp{}),T\Perp{})$ called a \newterm{tilting cotorsion pair}, here $\Tcal$ is called the \newterm{tilting class}. Two tilting modules $T,T'$ give rise to the same tilting cotorsion pair (or equivalently, the same tilting class) precisely when $\Add(T) = \Add(T')$, and in this situation we call them \newterm{equivalent}. In fact, $\Add(T)$ determines the tilting cotorsion pair in the sense that the class $\Acal$ consists precisely of modules admitting a finite $\Add(T)$-coresolution, $\Tcal$ consists precisely of modules admitting an $\Add(T)$-resolution, and we have $\Add(T) = \Acal \cap \Tcal$. We refer to \cite[\S 13]{GT12} for details. 

\subsection{}\label{ss:mintilting} A crucial result of Bazzoni-Herbera and Bazzoni-Šťovíček \cite{BH08,BS07} asserts that tilting cotorsion pairs coincide precisely with the cotorsion pairs of finite type. Assume $\findim(R) < \infty$. Then the cotorsion pair $(\Perp{1}(\Pcal^f\Perp{1}),\Pcal^f\Perp{1})$ of \cref{ss:findim} is clearly the \emph{minimal} cotorsion pair of finite type, where the ordering is given by inclusion of the second constituents. Therefore, $\Tcal_{\mathsf{min}}=\Pcal^f\Perp{1}$ is the \emph{minimal} tilting class with respect to inclusion. As discussed in \cref{ss:findim}, we have $\Tcal_{\mathsf{min}} = \{M \in \Mod R \mid \rid_R(M) \leq 0\}$.

\subsection*{Commutative noetherian rings} From now on, let $R$ be a commutative noetherian ring. Let $\dim(R)$ denote its Krull dimension and $\Spec R$ its Zariski spectrum.

\subsection{}\label{ss:depth} 
Given a cochain complex $M$, we define $\depth_R(\pp,M) = \inf \RHom_R(R/\pp,M)$ and $\width_R(\pp,M) = -\sup (R/\pp \otimes_R^\mathbf{L} M)$. In the case $(R,\mm,k)$ is local we simply let $\depth_R(M) = \depth_R(\mm,M)$ and $\width_R(M) = \width_R(\mm,M)$, and as with all similar invariants, we often omit the subscript if the ring is clear from context. We let $\grade_R(M) = \inf\{i \mid \Ext_R^i(M,R) \neq 0\}$ and by convention $\grade(\pp) := \grade(R/\pp) = \depth(\pp,R)$ for $\pp \in \Spec R$. One always has $\grade(\pp) \leq \depth(R_\pp) \leq \height(\pp) := \dim(R_\pp)$ and both the inequalities may fail to be equalities in general. The equality $\depth(R_\pp) = \height(\pp)$ occurs precisely when the local ring $R_\pp$ is Cohen-Macaulay. The equality $\grade(\pp) = \depth(R_\pp)$ holds for all $\pp \in \Spec R$ if and only if $R$ is an almost Cohen-Macaulay ring, see \cite[Lemma 3.1]{CFF02}.

\subsection{}\label{ss:classification} Angeleri-Hügel, Pospíšil, Šťovíček, and Trlifaj \cite[Theorem 4.2]{AHPS14} gave a full classification of tilting cotorsion pairs over a commutative noetherian ring. Here, we follow an exposition explained in \cite[Remark 5.10]{HNS}. We call a function $\ef: \Spec R \to \Zbb$ \newterm{characteristic} if the following hold:
\begin{enumerate}
  \item $\ef$ is order-preserving, that is, $\ef(\pp)\leq \ef(\qq)$ whenever $\pp \subseteq \qq$ in $\Spec R$, 
  \item we have $0 \leq \ef \leq \grade$,
  \item there is an $n \geq 0$ such that $\ef \leq n$ (this condition is superfluous if $\dim(R) < \infty)$.
\end{enumerate}   
Then there is a bijective correspondence between characteristic functions $\ef$ on $\Spec R$ and tilting cotorsion pairs $(\Acal,\Tcal)$ in $\Mod R$. Here, $\ef$ is sent to $(\Acal_\ef,\Tcal_\ef)$ with the tilting class $\Tcal_{\ef} = \{M \in \Mod R \mid \width_R(\pp,M) \geq \ef(\pp) ~\forall \pp \in \Spec R\}$. This correspondence restricts to one between tilting modules of projective dimension at most $n$ (= \newterm{$n$-tilting modules}) and characteristic functions $\ef$ with $\ef \leq n$.

\subsection{}\label{ss:RG} By classical results of Bass \cite{B62} and Raynaud-Gruson \cite{RG71}, $\Findim(R) = \dim(R)$, \cite[Théor\`{e}me 3.2.6]{RG71}. Assume now that $\dim(R)<\infty$. Then any flat $R$-module belongs to $\Pcal$ \cite{Jen70}, \cite[Corollaire 3.2.7]{RG71}, and therefore $\Pcal$ coincides with the class $\Fcal$ of all modules of finite flat dimension. It follows that $\Pcal$ can be described as the class of all modules of flat dimension bounded by $\dim(R)$. In symbols, we have $\Pcal_{\dim(R)} = \Pcal = \Fcal = \Fcal_{\dim(R)}$. Finally, $R$ being noetherian ensures that $\Fcal_{\dim(R)}$ is a \newterm{definable class}, that is, a subcategory of $\Mod R$ closed under direct limits, products, and pure submodules.

\subsection{}\label{ss:rid} Let $R$ be a commutative noetherian ring of finite Krull dimension. Among the characteristic functions $\ef: \Spec R \to \Zbb$ there is always the maximal choice of the grade function $\grade: \pp \mapsto \grade_R(\pp)$. The corresponding tilting cotorsion pair $(\Acal_{\grade},\Tcal_{\grade})$ with $\Tcal_{\grade} = \{M \in \Mod R \mid \width_R(\pp,M) \geq \grade(\pp) ~\forall \pp \in \Spec R\}$ is therefore precisely the minimal tilting cotorsion pair $(\Perp{1}(\Pcal^f\Perp{1}),\Pcal^f\Perp{1})$. In particular, we have using \cref{ss:mintilting}:
\begin{equation}\label{E:Tgrade}
  \Tcal_{\grade} = \{M \in \Mod R \mid \rid_R(M) \leq 0\}.
\end{equation}
By \cite[Proposition 5.3]{CFF02}, we can compute the small restricted injective dimension via the formula 
\begin{equation}\label{E:rid}
\rid_R(M) = \sup \{\grade_R(\pp) - \width (\pp, M) \mid \pp \in \Spec R\},
\end{equation} 
which also implies \cref{E:Tgrade} by comparing it directly with the classification \cref{ss:classification}.
\section{The minimal tilting class over a Cohen-Macaulay ring}\label{s:tiltingcm}
The aim of this section is to study the minimal tilting class $\Tcal_{\mathsf{min}}$ under the assumption that $R$ is Cohen-Macaulay of finite Krull dimension. As discussed in \cref{ss:depth}, this assumption ensures $\grade_R = \height_R$, and so in view of \cref{ss:rid} we have $\Tcal_{\mathsf{min}}=\Tcal_{\grade} = \Tcal_{\height} = \{M \in \Mod R \mid \width_R(\pp,M) \geq \height(\pp) ~\forall \pp \in \Spec R\}$.

\subsection{}\label{ss:Chouinard} 
Introduced in \cite{Cho76}, the \newterm{Chouinard invariant} is defined as 
$$\CH_R(M) = \sup \{\depth R_\pp - \width_{R_\pp}(M_\pp) \mid \pp \in \Spec R\}.$$ 
Note that $\CH_R(0) = - \infty$ and $\CH_R(M)$ is always bounded above by $\dim(R)$. The Chouinard invariant refines the injective dimension in the sense that one always has $\CH_R(M) \leq \id_R(M)$ and this becomes an equality whenever $M$ is cohomologically bounded and $\id_R(M) < \infty$, see \cite{Yas98}. 

Even for $R$ Cohen-Macaulay, one cannot expect the equality $\width(\pp,M) = \width(M_\pp)$ to hold in general. Nevertheless, the two invariants we defined in terms of these values coincide.
\begin{prop}\label{CM-rid}
  Let $R$ be a commutative noetherian ring of finite Krull dimension. For any cohomologically bounded complex $M$ we have 
  $$\sup\{\height(\pp)-\width(M_\pp) \mid \pp \in \Spec R\} =$$ 
  $$=\sup\{\height(\pp)-\width(\pp,M) \mid \pp \in \Spec R\}.$$
  As a consequence, if $R$ is in addition Cohen-Macaulay, then $\CH_R(M) = \rid_R(M)$.
\end{prop}
\begin{proof}
  We always have $\width_R(\pp,M) \leq \width(M_\pp)$ \cite[Corollary 4.12]{CFF02} for any $\pp \in \Spec R$, and so the left-hand side is always smaller or equal to the right-hand side. In order to show the other inequality, we will prove that $\Tor_i^R(R/\pp,M) = 0$ whenever $\Tor_i^R(k(\pp),M) = 0$ by a backwards induction on $\dim(R/\pp)$. If $\dim(R/\pp) = 0$ then $\pp$ is a maximal ideal, so $R/\pp = k(\pp)$, and there is nothing to prove. The short exact sequence $0 \to R/\pp \to k(\pp) \to L \to 0$ induces a piece of the long exact sequence:
  $$\Tor_{i+1}^R(L,M) \to \Tor_i^R(R/\pp,M) \to \Tor_i^R(k(\pp),M).$$
  For each $\qq \in V(\pp) \setminus \{\pp\}$, we have $\height(\qq) > \height(\pp)$ and $\dim(R/\qq) < \dim(R/\pp)$, and so the induction hypothesis applies and yields $\Tor_{i+1}^R(R/\qq,M) = 0$. Since $\Supp(L) \subseteq V(\pp) \setminus \{\pp\}$, we have $\Tor_{i+1}^R(L,M) = 0$. Since $\Tor_i^R(k(\pp),M) = 0$ by the assumption on $M$, we are done by the exact sequence above.

  Now assume $R$ is Cohen-Macaulay. Then we have $\grade(\pp) = \depth(R_\pp) = \height(\pp)$ for all $\pp \in \Spec R$, and so the claim implies the left-hand side is equal to $\CH_R(M)$ and the right-hand side to $\rid_R(M)$, see \cref{E:rid}.
\end{proof}
\begin{rmk}
  \cref{CM-rid} can fail for a non-Cohen-Macaulay ring, and in fact there is no inequality between $\CH_R$ and $\rid_R$ in general.
  
  By \cite[Corollary 5.9]{CFF02}, if $R$ is a local ring such that $\dim(R) > \depth R +1$, then there is a module $M$ with $\id_R(M) = \dim(R) - 1$, so $\id_R(M) = \CH_R(M)$, but $\rid_R(M) < \id_R(M)$. On the other hand, it can also happen that $\CH_R(M) < \rid_R(M)$. Indeed, let $(R,\mm)$ be a 1-dimensional local ring which is not Cohen-Macaulay, and let $M$ be an $R$-module satisfying $\Supp(M) = \{\mm\}$ and $R/\mm \otimes_R M = 0$. One can always take $M$ to be the first local cohomology module $H_\mm^1(R)$ of $R$, see \cref{ss:filtration}. Then $\CH_R(M) = \depth(R) + \sup(R/\mm \otimes_R^\mathbf{L} M) < 0$, while $\rid_R(M)$ is always non-negative whenever $M \neq 0$.
  
  Similarly, it can happen that $\CH_R(M) \neq \Rid_R(M)$, see \cite[Remark 5.12]{CFF02}.
\end{rmk}
Combining \cref{E:Tgrade} with \cref{CM-rid}, the minimal tilting class $\Tcal_{\height}$ can be described using the Chouinard invariant. 
\begin{cor}\label{T-descr}
  Let $R$ be a Cohen-Macaulay ring of finite Krull dimension. We have $\Tcal_{\height} = \{M \in \Mod R \mid \CH_R(M) \leq 0\}$.
\end{cor}
\subsection{} The tilting module inducing the minimal tilting class $\Tcal_{\height}$ has been explicitly constructed in \cite{HNS}. Let $\RGamma_{\pp}: \D(R) \to \D(R)$ denote the local cohomology functor associated to the support $V(\pp) \subseteq \Spec R$. For each prime ideal $\pp$ let us fix the notation $T(\pp) = H_\pp^{\height(\pp)}(R_\pp)$; here we use the standard symbol $H_\pp^{i}(M) = H^i\RGamma_\pp(M)$ for the $i$-th local cohomology at $\pp$. Recall that $T(\pp)$ is isomorphic to $\RGamma_\pp R_\pp [\height(\pp)]$ in $\D(R)$ when $R$ is Cohen-Macaulay, see e.g. \cite[Theorem 10.35]{24h}.
\begin{theorem}\cite[Corollary 4.9, Remark 2.7]{HNS}
  Let $R$ be a Cohen-Macaulay ring of finite Krull dimension. The module $T_{\height} = \bigoplus_{\pp \in \Spec R}T(\pp)$ is a tilting module inducing the minimal tilting cotorsion pair $(\Acal_{\height},\Tcal_{\height})$. 
\end{theorem}
\begin{rmk}
  Let $E(R/\pp)$ denote the indecomposable injective over $\pp \in \Spec R$. Then we have $T(\pp) \cong E(R/\pp)$ if and only if $R_\pp$ is a Gorenstein ring. Therefore, if $R$ is (locally) Gorenstein then $T_{\height} \cong \bigoplus_{\pp \in \Spec R}E(R/\pp)$, recovering \cite[Example 5.7]{AH13}. In this case, it is known that $\Tcal_{\height}$ is precisely the class $\Gcal\Ical_0$ of Gorenstein injective $R$-modules and $\Acal_{\height} = \Pcal$ \cite[Example 9.3]{BH09}.

  If $R$ is not Gorenstein, then it is more difficult to check that $\bigoplus_{\pp \in \Spec R}T(\pp)$ is a tilting module. The main problem is to check the self-orthogonality condition (T2) of \cref{ss:tilting} here, which is trivial in the Gorenstein case. This was done in a larger generality in \cite[Theorem 1.1]{HNS}, see \cite[Remark 4.11]{HNS} for the Cohen-Macaulay case relevant for us. We remark that if $R$ admits a dualizing module, then checking condition (T2) can be reduced to Ext-orthogonality of injectives using the infinite completion of Grothendieck duality due to Iyengar and Krause \cite{IK06}, this is explained in \cite[\S 3]{HNS}. In the absence of a dualizing module, a more technical proof is required \cite[\S 4]{HNS}, although the most difficult argument using transfinite cofiltrations can be skipped under the assumption of $\dim(R) < \infty$.
\end{rmk}
\subsection{Canonical filtration}\label{ss:filtration} Our first step is to prove a deconstruction result for modules in $\Tcal_{\height}$ which extends the canonical filtrations of Gorenstein injectives over Gorenstein rings due to Enochs and Huang \cite{E11}. Here, $R$ is assumed to be a Cohen-Macaulay ring of finite Krull dimension.

\begin{lemma}\label{unique-tor}
  If $M \in \Tcal_{\height}$ then $\Tor_i^R(T(\pp),M) = 0$ whenever $i \neq \height(\pp)$ for all $\pp \in \Spec R$.
\end{lemma}
\begin{proof}
  Since $T(\pp)$ is an $R_\pp$-module supported on $\{\pp R_\pp\}$, it is filtered by copies of $k(\pp)$, and thus we have the vanishing $\Tor_i^R(T(\pp),M) = 0$ for all $i< \height(\pp)$ by \cref{T-descr}. Since $T(\pp)$ is an $R_\pp$-module of finite flat dimension, we have $\fd_{R_\pp} T(\pp) \leq \dim(R_\pp) = \height(\pp)$, see \cref{ss:RG}. As $\Tor_i^R(T(\pp),M) = \Tor_i^{R_\pp}(T(\pp),M_\pp)$, we obtain the vanishing for $i> \height(\pp)$.
\end{proof}
\begin{lemma}\label{Mp}
  For any $M \in \Tcal_{\height}$, $\RGamma_{\pp}(M_\pp)$ is isomorphic in $\D(R)$ to an $R$-module $M(\pp) \in \Tcal_{\height}$.
\end{lemma}
\begin{proof}
  By \cref{unique-tor}, $\RGamma_{\pp}(M_\pp) \cong T(\pp)[-\height(\pp)] \otimes_R^\mathbf{L} M$ is isomorphic in $\D(R)$ to an $R$-module $M(\pp)$ in $\D(R)$, it remains to show that $M(\pp) \in \Tcal$. This follows from \cref{T-descr}, because $M(\pp) \otimes_R^\mathbf{L} k(\qq) \cong M \otimes_R^\mathbf{L} \RGamma_\pp R_\pp \otimes_R^\mathbf{L} k(\qq)$ is equal to zero if $\qq \neq \pp$ or to $M \otimes_R^\mathbf{L} k(\pp)$ if $\qq = \pp$, and so $\Tor_i^R(k(\pp),M(\pp)) = 0$ for any $i<\height(\pp)$ using $M \in \Tcal$.
\end{proof}
Let $W \subseteq \Spec R$ be a specialization closed subset, then the local cohomology with support on $W$ is the right derived functor $\RGamma_W(X): \D(R) \to \D(R)$ of the torsion functor $\Gamma_W: \Mod R \to \Mod R$ with respect to the hereditary torsion class $\{M \in \Mod R \mid \Supp(M) \subseteq W\}$. It follows that $\RGamma_W$ is the Bousfield localization functor away from the localizing subcategory $\{X \in \D(R) \mid \supp(M) \subseteq W\}$, where $\supp(M) = \{\pp \in \Spec R \mid k(\pp) \otimes_R^\mathbf{L} M \neq 0\}$ is the cohomological support. If $W_1 \subseteq W_0$ are two specialization closed subset then there is a canonical triangle for any $X \in \D(R)$: $\RGamma_{W_1}X \to \RGamma_{W_0}X \to X' \xrightarrow{+}$, where $\supp(X') \subseteq W_0 \setminus W_1$. Now assume that $\dim(W_0 \setminus W_1) \leq 0$, or in other words, there are no $\pp,\qq \in W_0 \setminus W_1$ such that $\pp \subsetneq \qq$. Then it follows that the object $X'$ from the triangle above is of the form $X' = \bigoplus_{\pp \in W_0 \setminus W_1}\RGamma_\pp X_\pp$. For details, see e.g. \cite[Remark 4.2]{HNS}.
\begin{theorem}\label{filtration}
  Let $R$ be a Cohen-Macaulay ring of finite Krull dimension $d$. Any module $M \in \Tcal_{\height}$ admits a filtration $0 = M_{d+1} \subseteq M_d \subseteq M_{d-1} \subseteq M_{d-2} \subseteq \cdots \subseteq M_0 = M$ such that $M_i/M_{i+1}$ is isomorphic to a direct sum $\bigoplus_{\height(\pp) = i}M(\pp)$, where $M(\pp)$ are the $\pp$-torsion and $\pp$-local modules belonging to $\Tcal_{\height}$ of \cref{Mp}.
\end{theorem}
\begin{proof}
  Let $W_k = \{\pp \in \Spec R \mid \height(\pp) \geq k\}$. We prove by a backward induction on $k = d,d-1,\ldots,0$ that $\RGamma_{W_k}M$ is (quasi-isomorphic to) a module admitting a filtration $0 = M_{d+1} \subseteq M_d \subseteq M_{d-1} \subseteq M_{d-2} \subseteq \cdots \subseteq M_k = \RGamma_{W_k}M$ such that $M_i/M_{i+1}$ is isomorphic to a direct sum $\bigoplus_{\pp \in \Spec R, \height(\pp) = i}M(\pp)$.
  
  If $k=d$, then $\RGamma_{W_d} M$ is supported only on maximal ideals and thus already $\RGamma_{W_k} M = \bigoplus_{\pp \in \Spec R, \height(\pp) = d}M(\pp)$. In the induction step, consider the triangle $\RGamma_{W_{k+1}} M \to \RGamma_{W_k} M \to M' \xrightarrow{+}$. As explained in the paragraph above, we have $M' = \bigoplus_{\pp, \height(\pp) = k} \RGamma_{\pp} M_\pp \cong \bigoplus_{\pp, \height(\pp) = k}M(\pp)$. Then all three components of the triangle are modules, thus the triangle is in fact induced by a short exact sequence of modules, which finishes the induction.
\end{proof}

\begin{rmk}
  If $R$ is a Gorenstein ring then \cref{filtration} recovers the canonical filtration of Gorenstein injectives result of Enochs and Huang \cite[Theorem 3.1]{E11}.
\end{rmk}

\subsection{Product-completeness of $T$}
\begin{prop}\label{AddT}
  We have $\Acal_{\height} \cap \Tcal_{\height} = \Add(T)$ is equal to $\Pcal \cap \Tcal_{\height}$.
\end{prop}
\begin{proof}
  Since $\Add(T_{\height}) = \Acal_{\height} \cap \Tcal_{\height}$ (see \cref{ss:tilting}) and $\Acal_{\height} \subseteq \Pcal$, clearly $\Add(T_{\height}) \subseteq \Pcal \cap \Tcal_{\height}$. For the other inclusion, let $M \in \Pcal \cap \Tcal_{\height}$, and we need to show that $M \in \Acal_{\height}$. Recall that $\Acal_{\height}$ is closed under coproducts and extensions. Since $M \in \Pcal$, also $M(\pp) \in \Pcal$ because $M(\pp) \cong M \otimes_R^\mathbf{L} \RGamma_\pp R_\pp$ and $\RGamma_\pp R_\pp$ is isomorphic in $\D(R)$ to a bounded complex of flat modules. Then $M(\pp) \in \Pcal \cap \Tcal_{\height}$ by \cref{Mp}. By the existence of the canonical filtration of \cref{filtration} and the above discussed closure properties of $\Acal_{\height}$, we can without loss of generality assume that $M = M(\pp)$ for some $\pp \in \Spec R$, or in other words, $M \cong \RGamma_\pp M_\pp$ in $\D(R)$. 
  
  Since $M \in \Pcal$, $M$ is of finite projective dimension also as an $R_\pp$-module, and there is a resolution $0 \to P^{-\height(\pp)} \to P^{-\height(\pp) + 1} \to \cdots \to P^0 \to M \to 0$ of length $\height(\pp)$ where $P^i$ is a projective $R_\pp$-module for each $i=-\height(\pp),-\height(\pp)+1,\cdots,0$. Applying $- \otimes_R T(\pp)$ to the truncated resolution, we obtain a complex $N^{-\height(\pp)} \to N^{-\height(\pp) + 1} \to \cdots \to N^{0}$ where $N^i = P^i \otimes_R T(\pp) \in \Add(T_{\height})$. By \cref{unique-tor}, this complex is exact in all degrees $i$ apart from $i = -\height(\pp)$, and the cohomology in degree $-\height(\pp)$ is isomorphic to 
  $$\Tor_{\height(\pp)}^R(T(\pp),M) = H^{-\height(\pp)}(T(\pp) \otimes_R^\mathbf{L} M) \cong H^0(\RGamma_\pp (M_\pp)) = M(\pp) =
   M.$$ We showed that $M$ has a finite coresolution by modules from $\Add(T_{\height})$, and therefore $M$ belongs to $\Acal_{\height}$ by \cite[Proposition 13.13]{GT12}.
\end{proof}
\subsection{}\label{ss:product-complete} For a reference about concepts from the theory of purity used in what follows, we refer the reader to \cite{Prest} or \cite{GT12}. A module $M$ is called \newterm{product-complete} if $\Add(M)$ is closed under products.
\begin{prop}\label{prod-compl}
  Let $M$ be a product-complete module. Then:
  \begin{enumerate}
    \item[(i)] $M$ is $\Sigma$-pure-injective, that is, any module in $\Add(M)$ is pure-injective,
    \item[(ii)] $\Add(M) = \Prod(M)$.  
  \end{enumerate}
  If $T$ is a tilting module inducing a cotorsion pair $(\Acal,\Tcal)$ then the following are equivalent:
  \begin{enumerate}
    \item[(1)] $T$ is product-complete,
    \item[(2)] $\Acal$ is definable,
    \item[(3)] $\Add(T)$ is definable.
  \end{enumerate}
\end{prop}
\begin{proof}
  $(i)$: Follows directly from \cite[Lemma 2.32(c)]{GT12}.

  $(ii):$ By the definition, $\Prod(M) \subseteq \Add(M)$. For any set $X$ consider the natural map $M^{(X)} \to M^X$. This is a pure monomorphism, and so this map splits by $(i)$. This shows $\Add(M) \subseteq \Prod(M)$.

  The equivalence of $(1)$ and $(2)$ for a tilting module $T$ is proved in \cite[Proposition 13.56]{GT12}. Since $\Add(T) = \Acal \cap \Tcal$ and $\Tcal$ is definable \cite[Corollary 13.42]{GT12}, $(2) \implies (3)$. On the other hand, $(3)$ implies that $\Add(T)$ is closed under products, which amounts to $(1)$.
\end{proof}

\begin{cor}\label{T-prod-compl}
  The module $T_{\height}$ is product-complete, and thus both $\Add(T_{\height})$ and $\Acal_{\height}$ are definable subcategories of $\Mod R$. As a consequence, $\Tcal_{\height}$ is an enveloping class (see \cite[\S 5]{GT12}).
\end{cor}
\begin{proof}
  By \cref{AddT}, $\Add(T_{\height})$ is an intersection of two definable subcategories $\Pcal$ and $\Tcal_{\height}$, and thus it is itself definable. The rest follows from \cref{prod-compl} and \cite[Theorem 7.2.6]{GT12}.
\end{proof}
\subsection{Finite type of $\Pcal$} 
\begin{lemma}\label{lemma-cp}
  Let $(\Xcal,\Ycal)$ be a complete hereditary cotorsion pair and $\Zcal$ a class of modules closed under extensions such that $\Xcal \subseteq \Zcal$ and $\Xcal \cap \Ycal = \Zcal \cap \Ycal$. Then $\Xcal = \Zcal$.
\end{lemma}
\begin{proof}
  Let $Z \in \Zcal$ and consider the exact sequence $0 \to Z \to Y^Z \to X^Z \to 0$ which exists by completeness of the cotorsion pair. Then $Y^Z \in \Ycal$, and by the assumptions, also $Y^Z \in \Zcal$. It follows that $Y^Z \in \Xcal$. Since the cotorsion pair is hereditary, $\Xcal$ is closed under kernels of epimorphisms and so $Z \in \Xcal$.
\end{proof}

\begin{theorem}\label{T:finite-type}
  Let $R$ be a commutative noetherian ring of finite Krull dimension. The following are equivalent:
  \begin{enumerate}
    \item[(i)] $R$ is Cohen-Macaulay,
    \item[(ii)] the cotorsion pair $(\Pcal,\Pcal\Perp{1})$ is of finite type. 
  \end{enumerate}
\end{theorem}
\begin{proof}
  $(i) \implies (ii):$ This follows directly from \cref{AddT} and \cref{lemma-cp} applied to the cotorsion pair $(\Acal_{\height},\Tcal_{\height})$ and $\Pcal$.

  $(ii) \implies (i):$ Notice first that if $(ii)$ is true for $R$, then the same also applies to each local ring $R_\pp$. Indeed, any $R_\pp$-module of finite projective dimension is of the form $M \otimes_R R_\pp$ for some $M \in \Pcal$, this is because any $R_\pp$-module of finite projective dimension is necessarily of finite projective dimension also as an $R$-module, see \cref{ss:RG}. Recall from \cref{ss:findim} that $(ii)$ is equivalent to any $M \in \Pcal$ being a direct summand in a $\Pcal^f$-filtered module. Assume first that $M= \bigcup_{\alpha \leq \lambda}M_\alpha$ is an expression of $M$ as a filtration of modules from $\Pcal^f$, that is: $(M_\alpha \mid \alpha \leq \lambda)$ is a continuous chain of submodules of $M$ such that $M_{\alpha + 1}/M_\alpha \in \Pcal^f$ for each ordinal $\alpha < \lambda$, $M_0 = 0$, and $M_\lambda = M$. Then tensoring this chain with $R_\pp$ yields the desired filtration of $M \otimes_R R_\pp$ by modules from $\mod{R_\pp}$ of finite projective dimension. It follows that any $R_\pp$-module of finite projective dimension is a direct summand in a module admitting such a filtration.

  By the previous paragraph, we may assume that $R$ is a local ring. Then we have $\findim(R) = \depth(R)$ by the Auslander-Buchsbaum formula \cite[Theorem 8.13]{24h}. By definition, $R$ not being Cohen-Macaulay amounts to $\depth(R) < \dim(R)$, and so any module which is a direct summand of a $\Pcal^f$-filtered module is of projective dimension strictly smaller then $\dim(R)$. On the other hand, by \cite[Proposition 5.4]{B62} there is $M \in \Pcal$ with $\pd_R(M) = \dim(R)$, which yields a contradiction with $(ii)$.
\end{proof}
\begin{rmk}
  Apart from the Gorenstein case, the implication $(i) \implies (ii)$ of \cref{T:finite-type} was also known to hold for Cohen-Macaulay rings of Krull dimension one. This is a particular case of results about the finite type of modules of projective dimension at most one studied in \cite{BH09} by Bazzoni and Herbera; see \cite[Theorem 8.4]{BH09} in particular.
\end{rmk}
The following is the injective counterpart of an analogous statement proved in \cite[Theorem 5.22]{CFF02} for the notions of restricted projective dimensions.
\begin{cor}\label{Rid}
  Let $R$ be a commutative noetherian ring of finite Krull dimension. The following are equivalent:
  \begin{enumerate}
    \item[(i)] $\Rid_R(M) = \rid_R(M)$ for any cohomologically bounded $R$-complex $M$,
    \item[(ii)] $\Rid_R(M) = \rid_R(M)$ for any $R$-module $M$,
    \item[(iii)] $R$ is Cohen-Macaulay. 
  \end{enumerate}
  
\end{cor}
\begin{proof}
  Follows directly from \cref{T:finite-type} and \cref{l:fintype}.
\end{proof}
\subsection{Existence and uniqueness of product-complete tilting modules}
We will show that the tilting module $T_{\height}$ is, up to equivalence, the unique product-complete tilting module.
\begin{lemma}\label{prodcomp-localize}
  Let $T$ be a product-complete tilting module and $\qq \in \Spec R$ a prime ideal. Then $T_\qq$ is a product-complete tilting module in $\Mod R_\qq$. 
\end{lemma}
\begin{proof}
  By \cref{prod-compl}, $\Add(T)$ is a definable subcategory of $\Mod R$, and therefore it is closed under direct limits. It follows that $\Add(T_\qq)$ is a full subcategory of $\Add(T)$. In fact, $\Add(T_\qq) = \Add(T) \cap \Mod R_\qq$. Indeed, if $M$ is an $R_\qq$-module in $\Add(T)$, then $M \cong M \otimes_R R_\qq$ belongs to $\Add(T_\qq)$. Since both $\Add(T)$ and $\Mod R_\qq$ are closed under products in $\Mod R$, it follows that $\Add(T_\qq)$ is closed under products in $\Mod R$, and thus in $\Mod R_\qq$ as well.
\end{proof}
\begin{theorem}\label{T:unique-prodcomp}
  Let $R$ be a commutative noetherian ring. The following are equivalent:
  \begin{enumerate}
    \item[(i)] $R$ is Cohen-Macaulay of finite Krull dimension,
    \item[(ii)] there is a product-complete tilting module $T$ in $\Mod R$.
  \end{enumerate}
  Furthermore, if these conditions are satisfied, then $T$ is equivalent to the height tilting module $T_{\height}$ as tilting modules.
\end{theorem}
\begin{proof}
  The implication $(i) \implies (ii)$ is \cref{T-prod-compl}. 
  
  Let $\ef: \Spec R \to \Zbb$ be the characteristic function corresponding to the tilting class $T\Perp{}$ as in \cref{ss:classification}. Then $0 \leq \ef \leq \grade$. We claim that $\ef = \height$. Note that this already implies that $R$ is Cohen-Macaulay, because then $\height = \ef \leq \grade \leq \height$ and so $\height = \grade$. 

  To show this, let $\mm$ be the minimal among prime ideals such that $\ef$ restricted to $\Spec {R_\mm}$ is not the height function on $\Spec {R_\mm}$. Using \cref{prodcomp-localize}, this reduces the question to $(R,\mm)$ being a local ring and $\ef(\pp) = \height(\pp)$ for all prime ideals $\pp$ apart from the maximal ideal $\mm$. The claim is trivial if $\dim(R) = 0$. If $\dim(R) = 1$ then $\ef(\pp) = 0$ for all $\pp \in \Spec R$, and so $T$ is a projective generator. Since $T$ is product-complete, $R$ is artinian, a contradiction. If $\dim(R) = 2$, then $\ef(\pp) = 0$ for any minimal prime $\pp \in \Spec R$ and $\ef(\pp) = 1$ otherwise. Therefore, $T$ is a 1-tilting module, that is, $\pd_{R}T = 1$ (see \cref{ss:classification}). Since $T$ is product-complete, its induced tilting class is enveloping in $\Mod{R}$ (see \cref{T-prod-compl}). Therefore, \cite[Theorem 8.7]{BLG22} implies that $R/\pp$ is artinian for any $\pp$ non-minimal. This is a contradiction with $\dim(R)=2$.
  
  Assume finally that $\dim(R) > 2$ and put $k=\dim(R)-2$. Let $I$ be any ideal of $R$ generated by a regular sequence of length $k$. Then any prime ideal in $V(I)$ has height at least $k$. By the description \cref{ss:classification}, we have $\Tor_i^R(R/I,T)=0$ for all $i<k$. Since $\pd_R(R/I)=k$ and $\height(I) = k$, it follows that the cohomology of $T \otimes_R^\mathbf{L} R/I$ vanishes outside of degree $-k$. By \cite[Theorem 4.2]{BHM}, $T \otimes_R^\mathbf{L} R/I$ is a \newterm{silting object} in $\D(R/I)$, and therefore $T \otimes_R^\mathbf{L} R/I[-k]$ is isomorphic in $\D(R/I)$ to a tilting $R/I$-module $\overline{T}=\Tor_{k}^R(R/I,T)$, see \cite[Remark 2.7]{HNS}.
  
  Since $R/I$ is a finitely generated $R$-module, the functor $\Tor_{k}^R(R/I,-)$ preserves products and restricts to a functor $\Add(T) \to \Add(\overline{T})$. Let us show that $\overline{T}$ is product-complete. For that, it is enough to show that for any collection of cardinals $\lambda_i, i \in I$, the product $\prod_{i \in I}\overline{T}^{(\lambda_i)}$ belongs to $\Add(\overline{T})$. Since $T$ is product-complete, the $R$-module $\prod_{i \in I}T^{(\lambda_i)}$ belongs to $\Add(T)$. But since $\Tor_{k}^R(R/I,\prod_{i \in I}T^{(\lambda_i)}) \cong \prod_{i \in I}\Tor_{k}^R(R/I,T)^{(\lambda_i)} = \prod_{i \in I}\overline{T}^{(\lambda_i)}$, the claim follows. The characteristic function $\overline{\ef}$ corresponding to the tilting $R/I$-module $\overline{T}$ can be computed as $\overline{\ef}(\overline{\qq}) = \ef(\qq) - k$ for any $\overline{\qq} \in \Spec {R/I}$, where $\qq \in \Spec R$ is the unique prime such that $\qq \in V(I)$ and $\qq/I = \overline{\qq}$, see \cite[Theorem 5.7]{BHM}. It follows that $\overline{\ef}$ values to 1 on every non-minimal prime ideal of $\Spec{R/I}$. The same proof as above applied to the 1-tilting $R/I$-module $\overline{T}$ shows that $\dim(R/I) \leq 1$, but at the same time $\dim(R/I) = 2$ by the choice of $I$ (see \cite[Lemma 10.60.14]{Stacks}), a contradiction.
  
  We proved that $\ef = \height$. Since $\ef$ is characteristic, it is bounded above, and therefore $\dim(R)<\infty$. Finally, $T$ is equivalent to $T_{\height}$ by \cref{ss:classification}.
\end{proof}
\section{Cohen-Macaulay Hom injective dimension}\label{s:cmhomdim}
A module over $R$ is said to be \newterm{Gorenstein injective} if it is a cocycle in an acyclic complex of injective modules $Q$ such that $\Hom_R(E,Q)$ is acyclic for any injective module $E$. A module is \newterm{Gorenstein flat} if it is a cocycle in an acyclic complex $F$ of flat modules such that $I \otimes_R F$ is acyclic for any injective module $I$. We denote by $\Gid_R$ and $\Gfd_R$ the Gorenstein injective and flat dimension of $R$-modules or $R$-complexes, see e.g. \cite{CKL17}. For $n \geq 0$, we let $\Gcal\Ical_n = \{M \in \Mod R \mid \Gid_R(M) \leq n\}$, and $\Gcal\Fcal_n = \{M \in \Mod R \mid \Gfd_R(M) \leq n\}$. A local ring $(R,\mm,k)$ is Gorenstein if and only if $\Gid_R(k)<\infty$ if and only if $\Mod R = \Gcal\Ical_{\dim(R)}$, and the same is true for the Gorenstein flat dimension. This extends the classical fact that a local ring $R$ is regular if and only if $\id_R(k)<\infty$ if and only if $\Mod R = \Ical_{\dim(R)}$. See \cite[\S 5, \S 6]{C00} for details about Gorenstein injective and flat dimensions.

There are notions of Cohen-Macaulay injective dimensions available in the literature which aim to extend the above situation to Cohen-Macaulay rings. Holm-J{\o}rgensen in \cite{HJ07} introduced the following version of Cohen-Macaulay injective dimension. Recall that a finitely generated module $C$ is called \newterm{semidualizing} if the homothety map $R \to \RHom_R(C,C)$ is an isomorphism. If in addition $\id_R C < \infty$, we call $C$ a \newterm{dualizing module}. Recall that if $R$ admits a dualizing module then it is Cohen-Macaulay of finite Krull dimension \cite[0AWS]{Stacks}, but the converse is not true, e.g. \cite[Proposition 3.1]{FR70} or \cite[Example 6.1]{Nis12}. We denote by $R \ltimes C$ the trivial extension of $R$ by $C$, which is a module-finite commutative $R$-algebra. Then the Cohen-Macaulay injective dimension in the sense of \cite{HJ07} is defined as 
$$\CMDid_R(M) = \inf \{\Gid_{R \ltimes C}(M) \mid C \text{ a semidualizing $R$-module}\}.$$ 

This notion satisfies several desiderata. By \cite[Corollary 4.10, 4.15]{SSY20}, we always have the inequalities $\CH_R(M) \leq \CMDid_R(M) \leq \Gid_R(M) \leq \id_R(M)$ for any module $M$. Furthermore, if any of these values is finite, then it is equal to all of the values to its left \cite{Cho76}, \cite{CS10}, \cite[Lemma 4.14]{SSY20}. The value $\CMDid_R(M)$ is finite for all $R$-modules $M$ if and only if $R$ (is Cohen-Macaulay and) admits a dualizing module \cite[Theorem 5.1]{HJ07}. Combined with \cref{T-descr} this yields immediately that the minimal tilting class consists precisely of the Cohen-Macaulay injective $R$-modules in this case.
\begin{cor} 
  Let $R$ be a Cohen-Macaulay ring admitting a dualizing module. Then $\Tcal_{\height} = \{M \in \Mod R \mid \CMDid_R(M) \leq 0\}$. 
\end{cor}

To be able to cover cases in which a dualizing module is absent, a different definition of Cohen-Macaulay dimensions is necessary. The following is recently due to Sahandi, Sharif, and Yassemi \cite{SSY20}, advancing the original approach of Gerko \cite{G01} for CM-dimension of finitely generated modules. A \newterm{CM-deformation} is a surjective local ring morphism $Q \to S$ such that $\grade_Q(S) = \Gfd_Q(S)$. Note that we always have $\grade_Q(S) \leq \Gfd_Q(S)$ (\cite[p. 1168]{G01}) and that $\grade_Q(S)$ is always a finite value. A \newterm{CM-quasi-deformation} is a diagram $R \to S \gets Q$ of local ring morphisms such that $R \to S$ is flat and $Q \to S$ is a CM-deformation. A typical example of such a diagram if $R$ is local Cohen-Macaulay is $R \to \widehat{R} \gets \widehat{R} \ltimes \omega_{\widehat{R}}$, where $R \to \widehat{R}$ is the completion map and $\omega_{\widehat{R}}$ is a dualizing module over $\widehat{R}$, which always exists by the Cohen structure theorem \cite[Theorem 29.4(ii)]{Mat89}. Note that for this particular CM-quasi-deformation, $\grade_{\widehat{R}\ltimes \omega_{\widehat{R}}}(\widehat{R}) = 0$, \cite[Lemma 3.6]{G01}. The Cohen-Macaulay injective dimension in the sense of \cite{SSY20} is defined as 
$$\CMSid_R(M) = $$
$$= \inf\{\Gid_Q(M \otimes_R S) - \Gfd_Q(S) \mid R \to S \gets Q \text{ is a CM-quasi-deformation}\}.$$
This notion always satisfies $\CMSid_R(M) \leq \Gid_R(M)$ and indeed, that $\CMSid_R(M)$ is finite for all modules if and only if $R$ is Cohen-Macaulay \cite[Theorem 3.4]{SSY20}. However, other desiderata are shown in \cite{SSY20} only for $M$ with finitely generated cohomology. In an attempt to remedy this, we suggest the following definition.
\begin{dfn}\label{d:CMI}
  For a local ring $R$ and any $R$-complex $M$ the \newterm{Cohen-Macaulay Hom} \newterm{injective dimension} is defined as follows:
  $$\CMHomid_R(M) = $$
  $$=\inf \left \{\Gid_Q(\RHom_R(S,M)) - \Gfd_Q(S) \middle\vert \begin{tabular}{ccc}$R \to S \gets Q$ is a\\ CM-quasi-deformation \end{tabular}\right \}.$$ When $R$ is not local, we extend the definition by setting 
  $$\CMHomid_R(M) = \sup \{\CMHomid_{R_\mm}(M_\mm) \mid \mm \text { maximal ideal}\}.$$
\end{dfn}
\begin{rmk}
  Our modified definition \cref{d:CMI} takes the same approach as the recent work of Sather-Wagstaff and Totushek \cite{SWT21} on complete intersection Hom injective dimension: We replaced the coefficient extension $- \otimes_R S$ with respect to the flat morphism $R \to S$ by the derived coefficient coextension functor $\RHom_R(S,-)$. The intuition here is rather straightforward: While $- \otimes_R S$ does not preserve even the ordinary injective dimension, $\RHom_R(S,-)$ preserves and reflects both injective \cite{CK16} and Gorenstein injective dimensions \cite{CS10}, see also \cref{r:fibres}.
\end{rmk}
Similarly as the Cohen-Macaulay flat dimension in \cite[Proposition 3.13]{SSY20}, our definition stays the same when we restrict to a special type of CM-quasi-deformations. Note that unlike in the case of $\CMSid_R$ in \cite[Proposition 3.12]{SSY20}, we do not need to restrict to finitely generated modules here. Recall that if $(R,\mm,k)$ is a local ring, the \newterm{closed fibre} of a local morphism $R \to S$ is the ring $S \otimes_R k$.
\begin{lemma}\label{artinian-fibre}
  Let $(R,\mm)$ be a local ring. For any cohomologically bounded complex $M$, we have:
  $$\CMHomid_R(M) = $$
  $$=\inf \left \{\Gid_Q(\RHom_R(S,M)) - \Gfd_Q(S) \middle\vert \begin{tabular}{ccc}$R \to S \gets Q$ is a\\ CM-quasi-deformation \\ such that the closed fibre of \\ $R \to S$ is artinian \end{tabular}\right \}.$$ 
\end{lemma}
\begin{proof}
  Let $R \to S \gets Q$ be a CM-quasi-deformation such that $\CMHomid_R(M) = \Gid_Q(\RHom_R(S,M)) - \Gfd_Q(S)$. Let $\overline{\PP} \in \Spec S$ be minimal such that $\overline{\PP} \cap R = \mm$, and let $\PP \in \Spec Q$ be the unique prime lying over $\overline{\PP} \in \Spec S$. Now $R \to S_{\overline{\PP}} \gets Q_{\PP}$ is a CM-quasi-deformation with $R \to S_{\overline{\PP}}$. Indeed, observe that $S_{\overline{\PP}} = S \otimes_Q Q_{\PP}$ implies $\Gfd_{Q_{\PP}}S_{\overline{\PP}} \leq \Gfd_Q(S)$ and $\grade_{Q_{\PP}}(S_{\overline{\PP}}) \geq \grade_Q(S)$. Since $Q \to S$ is a CM-deformation, we have 
  \begin{equation}\label{E:Gperf}
  \Gfd_{Q_{\PP}}S_{\overline{\PP}} \leq \Gfd_Q(S) = \grade_Q(S) \leq \grade_{Q_{\PP}}(S_{\overline{\PP}}) \leq \Gfd_{Q_{\PP}}S_{\overline{\PP}}.
  \end{equation}

  Similarly, it follows that $\RHom_R(S_{\overline{\PP}},M) \cong \RHom_Q(Q_{\PP},\RHom_R(S,M))$, and thus we get $\Gid_{Q_{\PP}}\RHom_R(S_{\overline{\PP}},M) \leq \Gid_Q(\RHom_R(S,M))$ using \cite[Theorem 1.7]{CS10}. Together with the previous paragraph,  we have $\Gid_{Q_{\PP}}\RHom_R(S_{\overline{\PP}},M) - \grade_{Q_{\PP}}(S_{\overline{\PP}}) \leq \Gid_Q(\RHom_R(S,M)) - \grade_Q(S) = \CMHomid_R(M)$, and therefore $\CMHomid_R(M)$ attains its value also when computed using the CM-quasi-deformation $R \to S_{\overline{\PP}} \gets Q_{\PP}$. By the choice of $\overline{\PP}$, $R \to S_{\overline{\PP}}$ has an artinian closed fibre.
\end{proof}
\begin{rmk}\label{R:closed-fibre}
  Let $R \to S$ be a flat local morphism with an artinian closed fibre (in fact, Cohen-Macaulay closed fibre is enough). Then $R$ is Cohen-Macaulay if and only if $S$ is Cohen-Macaulay, see \cite[p. 181, Corollary]{Mat89}.
\end{rmk}
Let us check that our definition still characterizes Cohen-Macaulay rings.

\begin{prop}\label{CMinj}
  The following are equivalent for a commutative noetherian ring:
  \begin{enumerate}
    \item[(i)] $R$ is Cohen-Macaulay,
    \item[(ii)] $\CMHomid_{R_\mm}(M_\mm) < \infty$ for all maximal ideals $\mm$ and all cohomologically bounded $R$-complexes $M$, 
    \item[(iii)] $\CMHomid_{R_\mm}(M_\mm) < \infty$ for all maximal ideals $\mm$ and all $R$-modules $M$,
    \item[(iv)]  $\CMHomid_{R_\mm}(k(\mm)) < \infty$ for all maximal ideals $\mm$.
  \end{enumerate}
\end{prop}
\begin{proof}
  Since Cohen-Macaulay-ness is checked on the stalks $R_\mm$ for maximal ideals $\mm$, the statement reduces to the case of a local ring $(R,\mm,k)$.

  $(i) \to (ii)$: Since $R$ is Cohen-Macaulay, so is $\widehat{R}$, and therefore there is a CM-quasi-deformation of the form $R \to \widehat{R} \gets Q$ with $Q$ regular (see \cite[Theorem 3.9]{G01}). We have $\pd_R \widehat{R} < \infty$ (\cref{ss:RG}), and so for any cohomologically bounded $M$, $\RHom_R(\widehat{R},M)$ has bounded cohomology. Since $Q$ is regular, we have that $\CMHomid_R(M) \leq \Gid_Q \RHom_R(\widehat{R},M) = \id_Q \RHom_R(\widehat{R},M)  < \infty$.

  $(ii) \to (iii)$: Trivial.

  $(iii) \to (iv)$: Trivial.

  $(iv) \to (i)$: By the assumption, there is a CM-quasi-deformation $R \to S \gets Q$ with $R \to S$ such that $\Gid_Q(\RHom_R(S,k))<\infty$. If $E(k)$ denotes the minimal injective cogenerator of $\Mod R$, we have $k \cong \Hom_R(k,E(k))$. It follows that $\RHom_R(S,k) = \Hom_R(S,k) \cong \Hom_R(S \otimes_R k,E(k))$. Using \cite[Theorem 3.6]{Holm}, we get $\Gfd_Q(S \otimes_R k) < \infty$. Let $K$ be the residue field of $S$, and note that $K \cong k^{(X)}$ as $R$-modules for some set $X$. Then we also have $\Gfd_Q(S \otimes_R K) = \Gfd_Q((S \otimes_R k)^{(X)}) < \infty$.
  
  Consider the canonical map $i: K \to S \otimes_R K$ obtained as $i = (R \to S) \otimes_R K$. Since $R \to S$ is a pure monomorphism in $\Mod R$ (see \cite[Lemma 35.4.8]{Stacks}), $i$ is a monomorphism in $\Mod K$, and thus it splits. Therefore, we have $\Gfd_Q(K) < \infty$. Since $K$ is also the residue field of $Q$, it follows that $Q$ is a Gorenstein ring by \cite[Theorem 17]{Masek}, and then $R$ is Cohen-Macaulay by the same argument as in \cite[Theorem 3.4]{SSY20}.
\end{proof}

\begin{lemma}\label{CHCMdef}
  Let $Q \to S$ be a CM-deformation. For any cohomologically bounded $S$-complex $M$ we have $\CH_Q(M) - \Gfd_Q(S) = \CH_S(M)$.
\end{lemma}
\begin{proof}
  The inequality $\CH_Q(M) - \Gfd_Q(S) \geq \CH_S(M)$ is proven in \cite[Proposition 4.8]{SSY20}. The other inequality actually also follows from the same proof. Indeed, let $\qq \in \Spec Q$ be such that $\CH_Q(M) = \depth(Q_\qq) - \width_{Q_\qq}(M_\qq)$. Clearly, we can choose $\qq \in \Supp(M) \subseteq \Supp(S)$, and let $\overline{\qq} \in \Spec S$ be the unique prime whose inverse image is $\qq$. The same computation as in the proof of \cite[Proposition 4.8]{SSY20} shows that 
    \begin{equation*}\begin{split}\CH_Q(M)  & = \depth(Q_\qq) - \width_{Q_\qq}(M_\qq) = \\ 
    & = \depth_{Q_\qq}(S_{\overline{\qq}}) + \Gfd_{Q_\qq}(S_{\overline{\qq}}) -\width_{Q_\qq}(M_\qq) = \\
   & = \depth(S_{\overline{\qq}}) - \width_{S_{\overline{\qq}}}(M_{\overline{\qq}}) + \Gfd_{Q_\qq}(S_{\overline{\qq}}) = \\
  & = \depth(S_{\overline{\qq}}) - \width_{S_{\overline{\qq}}}(M_{\overline{\qq}}) + \Gfd_{Q}(S) \leq \CH_S(M) +\Gfd_{Q}(S).
  \end{split}\end{equation*}
  We remark that, as in the proof of \textit{loc. cit.}, the second equality follows from the Auslander-Bridger formula \cite[Theorem 4.13]{AB69}, the third equality follows from surjectivity of $Q_\qq \to S_{\overline{\qq}}$ and \cite[Proposition 5.2(1)]{Iy99} (and its version for width, cf. \cite[\S 4]{FI03}), while the fourth equality is \cref{E:Gperf}.
\end{proof}

\begin{lemma}\label{CMidCHgen}
  Let $R$ be a local ring and $M$ a cohomologically bounded $R$-complex. Then:
  $$\CMHomid_R(M) = \inf \left \{\CH_S(\RHom_R(S,M)) \middle\vert \begin{tabular}{ccc}$R \to S \gets Q$ a CM-quasi-deformation\\ with $\Gid_Q(\RHom_R(S,M))<\infty$ \end{tabular} \right \}.$$
  If $\CMHomid_R(M) < \infty$, the infimum is attained at any CM-quasi-deformation $R \to S \gets Q$ such that $\CMHomid_R(M) = \Gid_Q(\RHom_R(S,M)) - \Gfd_Q(S)$.
\end{lemma}
\begin{proof}
  Let $R \to S \gets Q$ be a CM-quasi-deformation with $\Gid_Q(\RHom_R(S,M)) < \infty$. Then $\Gid_Q(\RHom_R(S,M)) = \CH_Q(\RHom_R(S,M))$ by \cite[Theorem C]{CS10}. By \cref{CHCMdef}, we have $\CH_Q(\RHom_R(S,M)) - \Gfd_Q(S) = \CH_S(\RHom_R(S,M))$. It follows that $\CMHomid_R(M) \leq \CH_S(\RHom_R(S,M))$. 
  
  The second claim follows since for such a CM-quasi-deformation $R \to S \gets Q$ we have $\CMHomid_R(M) = \CH_S(\RHom_R(S,M))$ by the previous computation.
\end{proof}
\begin{lemma}\label{chain}
  Let $R$ be a local ring. For any cohomologically bounded $R$-complex $M$, we have the inequality $\CMHomid_R(M) \leq \CMDid_R(M).$
  Furthermore, if $\CMDid_R(M) < \infty$ then $\CH_R(M) = \CMHomid_R(M) = \CMDid_R(M)$.
\end{lemma}
\begin{proof}
  Let $C$ be a semidualizing module such that $\CMDid_R(M) = \Gid_{R \ltimes C} M$. By \cite[Lemma 3.6]{G01}, $R \xrightarrow{=} R \gets R \ltimes C$ is a CM-quasi-deformation. By the definition, we thus have $\CMHomid_R(M) \leq \Gid_{R \ltimes C} M = \CMDid_R(M)$. 

  Now assume that $\Gid_{R \ltimes C} (M) < \infty$. By \cite[Lemma 4.14]{SSY20}, we get $\CMDid_R(M)=\Gid_{R \ltimes C} (M) = \CH_R(M)$. Let $R \to S$ be a flat local morphism. Then $C \otimes_R S$ is a semidualizing $S$-module \cite[Theorem 4.5]{FSW07} and $R \ltimes C \to (R \ltimes C) \otimes_R S \cong S \ltimes (C \otimes_R S)$ is a flat local morphism. Since 
  $$\RHom_R(S,M) \cong \RHom_{R \ltimes C}((R \ltimes C) \otimes_R S,M),$$ 
  it follows by \cite[Theorem 1.7]{CS10} that 
  $$\Gid_{R \ltimes C} M = \Gid_{(R \ltimes C) \otimes_R S} \RHom_R(S,M).$$ 
  Using \cite[Lemma 4.14]{SSY20} again, we get 
  $$\Gid_{(R \ltimes C) \otimes_R S} \RHom_R(S,M) = \Gid_{S \ltimes (C \otimes_R S)} \RHom_R(S,M) = \CH_S(\RHom_R(S,M)).$$ 
  In conclusion, $\CH_R(M) = \CH_S(\RHom_R(S,M))$ for all flat local morphisms $R \to S$. Choose a CM-quasi-deformation $R \to S \gets Q$ such that $\CMHomid_R(M) = \CH_S(\RHom_R(S,M))$ using \cref{CMidCHgen} and that $\CMHomid_R(M) \leq \CMDid_R(M) < \infty$, and then we obtain $\CMDid_R(M) = \CH_R(M) = \CMHomid_R(M)$.
\end{proof}

\begin{rmk}
  It is not clear to us whether \cref{chain} generalizes for non-local rings. The problem is that we do not know if the Holm-J{\o}rgensen dimension $\CMDid$ always satisfies the local-global principle.
\end{rmk}

\begin{quest} Does the refinement property 
  $$\CMHomid_R(M) < \infty \implies \CMHomid_R(M) = \CH_R(M)$$ 
  hold for any commutative noetherian ring $R$? In what follows, we are able to show this if $R$ is Cohen-Macaulay.
\end{quest}
\begin{lemma}\label{RidFF}
  Let $R$ be a commutative noetherian ring of finite Krull dimension. Let $R \to S$ be a faithfully flat ring homomorphism. Then for any cohomologically bounded $R$-complex $M$, we get an equality $\Rid_S(\RHom_R(S,M)) = \Rid_R(M)$ and an inequality $\rid_S(\RHom_R(S,M)) \geq \rid_R(M)$. 
  
  If $R$ is Cohen-Macaulay then we have $\rid_S(\RHom_R(S,M)) = \rid_R(M)$.

  If both $R$ and $S$ are Cohen-Macaulay then we have $\CH_S(\RHom_R(S,M)) = \CH_R(M)$.
\end{lemma}
\begin{proof}
  By the projective descent of Raynaud and Gruson \cite[Second partie]{RG71}, we have $\pd_R N < \infty$ if and only if $\pd_S N \otimes_R S < \infty$ for any $R$-module $N$. Since $\pd_R S < \infty$, we also have $\pd_S(L) < \infty$ if and only if $\pd_R(L) < \infty$ for any $S$-module $L$. The latter property together with the adjunction formula $\RHom_S(L,\RHom_R(S,M)) \cong \RHom_R(L,M))$ yields $\Rid_S(\RHom_R(S,M)) \leq \Rid_R(M)$.

  For the other inequality, let $N \in \Pcal$ be an $R$-module such that we have $n = \sup \RHom_R(N,M) = \Rid_R(M)$. We first claim that $\sup \RHom_R(N \otimes_R S,M) = \Rid_R(M)$. Recall e.g. from 
  \cite[Lemma 35.4.8]{Stacks} that the exact sequence 
  \begin{equation}\label{E:RtoS} 0 \to R \to S \to S/R \to 0 \end{equation} 
  is pure, and then all of its components are flat $R$-modules. Applying $\Hom_{\D(R)}(N \otimes_R^\mathbf{L} -,M)$ to \cref{E:RtoS}, we obtain an exact sequence $\Ext_R^n(N \otimes_R S,M) \to \Ext_R^n(N,M) \to \Ext_R^{n+1}(N\otimes_R S/R,M)$. Since $S/R$ is flat, $N\otimes_R S/R \in \Pcal$ (see \cref{ss:RG}), and so $\Ext_R^{n+1}(N\otimes_R S/R,M) = 0$ by the assumption. Therefore $\Ext_R^n(N \otimes_R S,M) \neq 0$ and the claim follows. Next, by adjunction we have $\RHom_R(N \otimes_R S,M) \cong \RHom_S(N \otimes_R S,\RHom_R(S,M))$, and thus it follows that $\Rid_S(\RHom_R(S,M)) \geq \Rid_R(M)$.

  The argument of the previous paragraph applied for $N \in \Pcal^f$ shows also the inequality $\rid_S(\RHom_R(S,M)) \geq \rid_R(M)$. If $R$ is Cohen-Macaulay, we have using \cref{Rid} and the above that $\rid_R(M) = \Rid_R(M) = \Rid_S(\RHom_R(S,M)) \geq \rid_S(\RHom_R(S,M))$. The final claim follows from the previous one and \cref{CM-rid}.
\end{proof}

\begin{prop}\label{CHCMid}
  Let $R$ be a Cohen-Macaulay ring. For any cohomologically bounded $R$-complex we have $\CH_R(M) = \CMHomid_R(M) = \rid_R(M) = \Rid_R(M)$.
\end{prop}
\begin{proof}
  The claim $\CH_R(M) = \CMHomid_R(M)$ clearly reduces to $R$ local. We have $\CMHomid_R(M) < \infty$ by \cref{CMinj}. In view of \cref{artinian-fibre} and \cref{CMidCHgen}, there is a CM-quasi-deformation $R \to S \gets Q$ with $R \to S$ having artinian closed fibre and such that $\CMHomid_R(M) = \CH_{S}(\RHom_R(S,M))$. Since both $R$ and $S$ are Cohen-Macaulay by \cref{R:closed-fibre}, we further have $\CH_{S}(\RHom_R(S,M)) = \CH_R(M)$ by \cref{RidFF}. Finally, $\CH_R(M) = \rid_R(M) = \Rid_R(M)$ by \cref{CM-rid,Rid}.
\end{proof}

For convenience, let us denote $\Ccal\Mcal\Ical_0 = \{M \in \Mod R \mid \CMHomid_R(M) \leq 0\}$. 

\begin{cor}\label{TCMid}
If $R$ is Cohen-Macaulay of finite Krull dimension then $\Tcal_{\height} = \Ccal\Mcal\Ical_0$. In particular, there is a cotorsion pair $(\Pcal,\Ccal\Mcal\Ical_0)$ and the class $\Ccal\Mcal\Ical_0$ is definable and enveloping.
\end{cor}
\begin{proof}
  Combine \cref{CHCMid} and \cref{T-descr}. The second claim follows from \cref{T-prod-compl,T:finite-type}.
\end{proof}

In the following proposition, we gather some further good properties of $\CMHomid$ over a Cohen-Macaulay ring analogous to those enjoyed by $\Gid$ over a Gorenstein ring.

\begin{prop}\label{CMid-props}
  Let $R$ be a Cohen-Macaulay ring. Then: 
  \begin{enumerate}
    \item[(i)] $\CMHomid_R(M) \leq \dim(R)$ for any $R$-module $M$,
    \item[(ii)] We have $\CMHomid_{R}(M_\pp) = \CMHomid_{R_\pp}(M_\pp) \leq \CMHomid_R(M)$ for any cohomologically bounded complex $M$ and any $\pp \in \Spec R$.
    \item[(iii)] Let $R \to S$ be a flat local morphism with a Cohen-Macaulay closed fibre. Then $\CMHomid_R(M) = \CMHomid_S(\RHom_R(S,M))$ for any cohomologically bounded complex $M$.
    \item[(iv)] Let $R$ be a local Cohen-Macaulay ring. Then we have $\CMHomid_R(M) = \CMDid_{\widehat{R}}(\RHom_{R}(\widehat{R},M))$ for any cohomologically bounded $R$-complex $M$. 
    \item[(v)] If $R$ is local and $M \neq 0$ is finitely generated $R$-module then $\CMHomid_R(M) = \depth(R)$. 
    \item[(vi)] If $R$ is local, we have $\CMHomid_R(M) = \CMDid_R(M)$ for any $R$-module $M$ or any cohomologically bounded $R$-complex if and only if $R$ admits a dualizing module.
  \end{enumerate}
  
\end{prop}
\begin{proof}
  $(i)$: By \cref{CHCMid}, we have $\CMHomid_R(M) = \CH_R(M) \leq \dim(R)$.

  $(ii)$: Observe that $\Tcal_{\height} \cap \Mod{R_\pp}$ is the minimal tilting class in $\Mod{R_\pp}$, and so $\Tcal_{\height} \cap \Mod{R_\pp} = \{M \in \Mod{R_\pp} \mid \CMHomid_{R_\pp}(M)\leq 0\}$. It follows directly from the fact that $\Tcal_{\height}$ is a definable subcategory that $\CMHomid_R(M) \leq 0 \implies \CMHomid_{R_\pp}(M_\pp) \leq 0 \iff \CMHomid_{R}(M_\pp) \leq 0$. Since $\CMHomid = \rid$ for both $R$ and $R_\pp$ by \cref{CHCMid}, it can be computed by taking $\Tcal_{\height}$-coresolutions (see \cref{ss:findim}), and the claim follows.

  $(iii):$ By \cref{CHCMid}, \cref{RidFF}, and \cref{R:closed-fibre}, we have $\CMHomid_R(M) = \CH_R(M) = \CH_S(\RHom_R(S,M)) = \CMHomid_S(\RHom_R(S,M))$.

  $(iv):$ By (iii), we have $\CMHomid_R(M) = \CMHomid_{\widehat{R}}(\RHom_{R}(\widehat{R},M))$. The equality of $\CMDid_{\widehat{R}}(\RHom_{R}(\widehat{R},M))$ with $\CMHomid_{\widehat{R}}(\RHom_{R}(\widehat{R},M))$ follows from (vi), because $\widehat{R}$ admits a dualizing module.

  $(v):$ By \cref{CHCMid}, $\CMHomid_R(M) = \rid_R(M)$, and then $\rid_R(M) = \depth(R)$ by \cite[Corollary 5.5]{CFF02}.

  $(vi)$: Recall from \cref{CMinj} that $\CMHomid_R(M)<\infty$ for any $M$. By \cite[Theorem 5.1]{HJ07}, $\CMDid_{R}(M)<\infty$ for every $M$ if and only if $R$ admits a dualizing module. Finally by \cref{chain}, $\CMHomid_R(M) = \CMDid_R(M)$ if and only if $\CMDid_R(M)<\infty$.
 \end{proof}

\begin{rmk} \label{r:fibres}
  \cref{CMid-props}(vi) gives a formula for computing $\CMHomid_R(M)$ over a local Cohen-Macaulay ring. Indeed, combined with \cite[Lemma 4.14]{SSY20}, we see that $\CMHomid_R(M) = \CMDid_{\widehat{R}}(\RHom_R(\widehat{R},M)) = \Gid_{\widehat{R} \ltimes \omega_{\widehat{R}}}(\RHom_R(\widehat{R},M))$.

  An analogous formula fails to hold for the notion $\CMSid$ of Cohen-Macaulay injective dimension from \cite{SSY20}. Let $R$ be a local Cohen-Macaulay ring such that there is a formal fibre $\widehat{R} \otimes_R k(\pp)$ for some $\pp \in \Spec R$ with Krull dimension $\dim(\widehat{R} \otimes_R k(\pp))>0$; such examples are abundant, see e.g. \cite{R91}. We consider $\Spec{\widehat{R} \otimes_R k(\pp)}$ naturally embedded into $\Spec{\widehat{R}}$ and for a $\PP \in \Spec{\widehat{R} \otimes_R k(\pp)}$, the residue field $k(\PP) = \widehat{R}_{\PP}/\PP\widehat{R}_{\PP}$ is a $k(\pp)$-module, so that as an $R$-module we have $k(\PP)$ is isomorphic to a coproduct of copies of $k(\pp)$. It follows that $\width_{\widehat{R}_{\PP}}(T(\pp) \otimes_R \widehat{R}) = -\sup(k(\PP) \otimes_{\widehat{R}}^\mathbf{L} T(\pp) \otimes_R \widehat{R}) = -\sup(k(\PP) \otimes_R^\mathbf{L} T(\pp)) = -\sup(k(\pp) \otimes_R^\mathbf{L} T(\pp))  = \width_{R_\pp}(T(\pp))$. Choosing $\PP$ as any non-minimal element of $\Spec{\widehat{R} \otimes_R k(\pp)}$, we get $\CH_{\widehat{R}}(T(\pp)) \otimes_R \widehat{R} \geq \depth(\widehat{R}_{\PP}) - \width_{\widehat{R}_{\PP}}(T(\pp) \otimes_R \widehat{R}) = \depth(\widehat{R}_{\PP}) - \width_{R_\pp}(T(\pp)) \geq \height(\PP) - \height(\pp) > 0$. 
  
  Since $\widehat{R}$ is Cohen-Macaulay with a dualizing module, we have $\CH_{\widehat{R}}(T(\pp) \otimes_R \widehat{R}) = \CMSid_{\widehat{R}}(T(\pp) \otimes_R \widehat{R}) = \CMHomid_{\widehat{R}}(T(\pp) \otimes_R \widehat{R}) > 0$, \cite[4.9, 4.13, 4.14]{SSY20}. On the other hand $\CMHomid_R(T(\pp)) = 0$, since $T(\pp) \in \Tcal_{\height}$. If $R$ itself admits a dualizing module then $\CMSid_R(T(\pp)) = \CMDid_R(T(\pp)) = \CMHomid_R(T(\pp)) = 0$ by \cref{CMid-props} and \cite[Corollary 4.15]{SSY20}.
\end{rmk}
\subsection{} We conclude the section by a generalization of the fact \cite[Theorem 4.1]{E11} that Gorenstein injective modules over Gorenstein rings are closed under taking tensor products.
\begin{cor}
  Let $R$ be a Cohen-Macaulay ring. For any $M, N \in \Ccal\Mcal\Ical_0$ we have $M \otimes_R N \in \Ccal\Mcal\Ical_0$.
\end{cor}
\begin{proof}
  We can assume that $R$ is local. Let $W = \{\pp \in \Spec R\mid \height(\pp)>0\}$, and consider the exact sequence $0 \to \Gamma_W(N) \to N \to N' \to 0$, where $\Gamma_W$ is the $W$-torsion functor. Since $M \in \Tcal_{\height}$, we have $R/I \otimes_R M = 0$ for any ideal $I$ such that $V(I) \subseteq W$. Because $\Supp (\Gamma_W(N)) \subseteq W$, $\Gamma_W(N) \otimes_R M = 0$. It follows that $N \otimes_R M \cong N' \otimes_R M$, so we can without loss of generality assume $\Gamma_W(N) = 0$.
  
  Now recall that since any module from $\Ccal\Mcal\Ical_0 = \Tcal_{\height}$ admits a canonical filtration by \cref{filtration}, we have $N \cong \bigoplus_{\qq \in \Spec R \setminus W}N_\qq$, and therefore we can without loss of generality assume that $N \cong N_\qq$ for some minimal prime $\qq$. But then $M \otimes_R N \cong (M \otimes_R N)_\qq$ is an $R_\qq$-module. It follows directly from the description of the minimal tilting class $\Tcal_{\height}$ of \cref{T-descr} that any $R_\qq$-module belongs to $\Tcal_{\height}$. 
\end{proof}
\section{The minimal cotilting class and Cohen-Macaulay flats}\label{s:cotilting}
In this section, we gather some dual results about the minimal cotilting class. It turns out that in this dual setting these are considerably easier to obtain. The definition of a \newterm{cotilting module} is dual to that of a tilting module. Namely, if $R$ is an associative unital ring, a left $R$-module $C$ is cotilting if $C \in \Ical$, $\Prod(C) \subseteq \Perp{}C$, and there is a short exact sequence $0 \to C_n \to \cdots \to C_1 \to C_0 \to W \to 0$ where all $C_i$'s belong to $\Prod(C)$ and $W$ is an injective cogenerator in the category $\lMod R$ of left $R$-modules. Any cotilting module $C$ induces the \newterm{cotilting cotorsion pair} $(\Perp{}C,(\Perp{}C)\Perp{})$. Two cotilting modules $C$ and $C'$ are equivalent if they induce the same cotilting cotorsion pair, or equivalently, if $\Prod(C) = \Prod(C')$. Let $(-)^+ = \Hom_{\Zbb}(-,\Qbb/\Zbb): \Mod R \to \lMod R$ denote the character duality functor. If $R$ is commutative noetherian then $T \mapsto T^+$ induces a bijection between the equivalence classes of tilting and cotilting $R$-modules \cite[Theorem 4.2]{AHPS14}. In this situation, let $\Tcal = T\Perp{}$ and $\Ccal = \Perp{}{(T^+)}$ be the induced tilting and cotilting class. Then $\Tcal$ and $\Ccal$ are \newterm{dual definable} classes. This means by definition that they are both definable classes and that we have the relations for any $M \in \Mod R$: $M \in \Tcal$ if and only if $M^+ \in \Ccal$, and analogously, $M \in \Ccal$ if and only if $M^+ \in \Tcal$. For more details, see \cite[\S 15, \S 16]{GT12}.

\subsection{} The \newterm{restricted flat dimension} of an $R$-module or $R$-complex $M$ was also introduced in \cite{CFF02} and is defined as $\Rfd_R(M) = \sup\{i \mid \Tor_i^R(\Fcal,M) \neq 0\}$. Unlike in the case of the restricted injective dimensions, one can always compute $\Rfd$ via a dual analog of the Chouinard invariant. That is, over any commutative noetherian ring $R$, we have the equality $\Rfd_R(M) = \sup\{\depth(R_\pp) - \depth_{R_\pp}(M_\pp) \mid \pp \in \Spec R\}$ for a cohomologically bounded $R$-complex $M$, see \cite[Theorem 2.4]{CFF02}.

\subsection{} Following \cite{SSY20}, the \newterm{Cohen-Macaulay flat dimension} over a local ring $R$ is defined as 
$$\CMSfd_R(M) =$$ 
$$= \inf\{\Gfd_Q(M \otimes_R S) - \Gfd_Q(S) \mid R \to S \gets Q \text{ is a CM-quasi-deformation}\}.$$
As with Cohen-Macaulay injectives, one has $\CMSfd_R(M) < \infty$ for all $R$-modules $M$ if and only if $R$ is Cohen-Macaulay \cite[Theorem 3.3]{SSY20}. For an arbitrary commutative noetherian ring $R$, we put $\Ccal\Mcal\Fcal_0 = \{M \in \Mod R \mid \CMSfd_{R_\mm}(M_\mm) \leq 0 ~\forall \mm \text{ maximal ideal}\}$. 

\subsection{} Let $R$ be a Cohen-Macaulay ring of finite Krull dimension. Denote by $\Ccal_{\height}$ the minimal cotilting class in $\Mod R$, that is, $\Ccal_{\height} = \{M \in \Mod R \mid \depth_R(\pp,M) \geq \height(\pp) ~\forall \pp \in \Spec R\}$, see \cite[Theorem 4.2]{AHPS14}. It follows directly from the general description of cosilting t-structures in $\D(R)$ \cite[2.15, 2.16]{HNS} that $\Ccal_{\height} = \{M \in \Mod R \mid \depth_{R_\pp}(M_\pp) \geq \height(\pp) ~\forall \pp \in \Spec R\}$. Recalling that the Cohen-Macaulay-ness of $R$ ensures $\depth_R(\pp,R) = \depth(R_\pp) = \height(\pp)$ for all $\pp \in \Spec R$, we have $\Ccal_{\height} = \{M \in \Mod R \mid \Rfd_R(M) \leq 0\}$.

Recall that if $R$ is a commutative noetherian ring with a dualizing complex then it is known that the classes $\Gcal\Ical_0$ and $\Gcal\Fcal_0$ of Gorenstein injective and flat modules are dual definable, $\Gcal\Fcal_0$ is covering, and $\Gcal\Ical_0$ is enveloping \cite[(2.6,3.3)]{HJ09}, \cite{EI15}.
\begin{prop}\label{cm-flat}
  Let $R$ be a Cohen-Macaulay ring, then $\Ccal_{\height} = \Ccal\Mcal\Fcal_0$. Therefore, $\Ccal\Mcal\Fcal_0$ is a covering class. In addition, the classes $\Ccal\Mcal\Fcal_0$ and $\Ccal\Mcal\Ical_0$ are dual definable.
\end{prop}
\begin{proof}
  The first claim follows from the equality $\Rfd_R \equiv \CMSfd_R$ which is proved in \cite[Corollary 4.2, Theorem 3.3]{SSY20}. The rest follows from \cref{TCMid}, \cite[Theorem 15.9]{GT12}, and the discussion above.
\end{proof}
\bibliographystyle{amsalpha}
\bibliography{bibitems}
\end{document}